\documentclass{acta_2000}
\usepackage{amssymb}
\usepackage{harvard}
\usepackage{psfig}

\newtheorem{thm}{Theorem}[section]
\newtheorem{lem}[thm]{Lemma}
\newtheorem{prop}[thm]{Proposition}

\newtheorem{algorithm}{Algorithm}[section]

\newcommand\be{\begin{equation}}
\newcommand\ee{\end{equation}}
\newcommand\half{\frac{1}{2}}

\newcommand\CC{{\mathbb C}}

\newcommand\OO{{\cal O}}

\newcommand\RR{{\mathbb R}}
\newcommand\ZZ{{\mathbb Z}}

\include{latex_macros}

\title[Preconditioned CG and RBFs]
{Preconditioned Conjugate Gradients, Radial Basis Functions
and Toeplitz Matrices}
\author[B.~J.~C.~Baxter]{B.~J.~C.~Baxter\\
Department of Mathematics, 
Imperial College\\
London SW7 2BZ, 
England\\ 
{\tt b.baxter@ic.ac.uk}  \\
{\tt www.ma.ic.ac.uk/$\sim$baxter}\\}

\begin{document}

\maketitle

\def\chptitle{Preconditioned conjugate gradients}


\def\pts{{(x_j)_{j=1}^n}}

\def\sjk{{\sum_{j,k=1}^n y_j y_k\,}}
\def\djk{{\| x_j - x_k \|}}
\def\xjk{{x_j-x_k}}
\def\yjk{{|\sum_{j=1}^n y_j e^{ix_j t}|^2}}
\def\trig{{|\sum_{j=1}^n y_j e^{ijt}|^2}}
\def\sjzd{{\sum_{j\in\Zd}}}
\def\skzd{{\sum_{k\in\Zd}}}

\def\rint{{\int_{-\infty}^{\infty}}}
\def\hint{{\int_0^\infty}}

\def\mud{{\mu_{d-1}}}
\def\Sdm1{{S^{d-1}}}

\def\ysqr{{\Vert y \Vert^2}}

\def\symbol{{\sum_{k\in\ZZ} |\phihat(t+2k\pi)| }}

\def\ainv{{A_n^{-1}}}
\def\aninv{{\Vert A_n^{-1} \Vert_2}}

\def\zsum{{\sum_{k=-\infty}^\infty}}

\def\hi{{(2\pi)^{-1}}}
\def\tint{{\hi\int_0^{2\pi}}}
\def\tintd{{(2\pi)^{-d}\int_{\Td}}}
\def\fhat{{\hat f}}
\def\ajxi{{\Bigl|\sum_{j\in\Zd} a_j \exp(ij\xi)\Bigl|^2}}
\def\Rdm0{{\Rd\setminus\{0\}}}
\def\ghat{{\hat g}}
\def\ytrigo{{\Bigl|\sum_{j\in\Zd} y_j \exp(ix_j \xi)\Bigl|^2}}
\def\ytrig2{{\Bigl|\sum_{j\in\Zd} y_j \exp(ij \xi)\Bigl|^2}}
\def\ynj{{y_j^{(n)}}}

\def\phichat{{\phihat_c}}
\def\chichat{{{\hat \chi}_c}}

\def\min{{\rm min}}


\def\pf{{\noindent{\it Proof.\ }}}
\def\qed{{\hfill$\square$\vskip 10pt}}

\def\Rd{{\RR^d}}
\def\Rn{{\RR^n}}
\def\ZZ{{{\cal Z} }}
\def\Z0{{\sum_{j=1}^n y_j = 0}}
\def\Zd{{\ZZ^d}}
\def\T1{{[0,2\pi]}}
\def\Td{{[0,2\pi]^d}}
\def\CC{{\cal C}}

\def\half{{1\over2}}
\def\OO{{\cal O}}
\def\phi{{\varphi}}
\def\phihat{{\hat\varphi}}


\def\nullx{\hfill}

\def\theta{{\vartheta}}
\def\psihat{\hat\psi}
\def\Rnn{\RR^{n \times n}}
\def\Image{\hbox{Im\ }}
\def\Sp{\hbox{Sp\ }}
\def\vol{\hbox{vol}}
\def\Cd{\CC^d}


\begin{abstract}
Radial basis functions provide highly useful and flexible interpolants
to multivariate functions. Further, they are beginning to be used in
the numerical solution of partial differential equations.
Unfortunately, their construction requires
the solution of a dense linear system. Therefore much attention has
been given to iterative methods. In this paper, we present a highly
efficient preconditioner for the conjugate gradient solution of the
interpolation equations generated by gridded data. Thus our method
applies to the corresponding Toeplitz matrices. The
number of iterations required to achieve a given tolerance 
is independent of the number of
variables.
\end{abstract}

\section{Introduction}
A radial basis function approximation has the form
\[
 s(x) = \sum_{j=1}^n y_j\, \phi(\|x-x_j\|), \qquad x \in \Rd, 
\]
where $\phi \colon [0,\infty) \to \RR$ is some given function,
$(y_j)_1^n$ are real coefficients, and the centres $(x_j)_1^n$ are points
in $\Rd$; the norm $\|\cdot\|$ will be Euclidean throughout this study.
For a wide class of functions $\phi$, it is known that the interpolation
matrix 
\[
A = (\phi(\|x_j-x_k\|))_{j,k=1}^n
\]
is invertible. This matrix is typically full, which fact has encouraged the
study of iterative methods. 
For example, highly promising results have been published in the use
of radial basis functions in collocation and Galerkin 
methods for the numerical solution of partial differential equations
(see \citeasnoun{FrankeSchaback} and \citeasnoun{Wendland}), but
direct solution limits their applicability to fairly small problems.
The use of the preconditioned conjugate
gradient algorithm was pioneered by \citeasnoun{DynLevinRippa}, and
some stunning results for scattered data were presented recently in 
\citeasnoun{FaulPowell}, although the rapid convergence described
there
is not fully understood. Therefore we study the highly
structured case when the data form a finite regular grid. The
conjugate gradient algorithm has been applied to Toeplitz matrices
with some success; see, for instance, \citeasnoun{Chan}. However,
since our matrices are usually {\em not} positive definite and often
possess elements that grow away from the diagonal, the preconditioners
of \citeasnoun{Chan} are not suitable. However, the matrices have the
property that their inverses tractable more tractable. Specifically, the
detailed study of the spectra of the associated Toeplitz operators
presented in Baxter (1992) and Baxter (1994) allows us to create
highly efficient preconditioners by inverting relatively small finite
sections of the bi-infinite symmetric Toeplitz operator, and this
construct is also easily understood via Toeplitz theory.

Let $n$ be a positive integer and let $A_n$ be the symmetric Toeplitz
matrix given by
\be
 A_n = \left(\phi(j-k)\right)_{j,k=-n}^n, \label{chapno1.1}
\ee
where $\phi \colon \RR \to \RR$ is either a Gaussian ($\phi(x) =
\exp(-\lambda x^2)$ for some positive constant $\lambda$) or a
multiquadric ($\phi(x) = (x^2 + c^2)^{1/2}$ for some real constant
$c$). In this paper we construct efficient preconditioners for the
conjugate gradient solution of the linear system
\be
A_n x = f, \qquad f \in \RR^{2n+1}, 
\label{chapno1.2}\ee
when $\phi$ is a Gaussian, or the augmented linear system
\begin{eqnarray} 
            A_n x + e y &=& f, \\
            e^T x &=& 0,
\label{chapno1.3}
\end{eqnarray}
when $\phi$ is a multiquadric. Here $e = [1, 1, \ldots, 1]^T 
\in \RR^{2n+1}$ and $y \in \RR$.
Section 2 describes the construction for the Gaussian and 
Section 3
deals with the multiquadric. Of course, we exploit the Toeplitz
structure of $A_n$ to perform a matrix-vector multiplication in $\OO(n
\log n)$ operations whilst storing $\OO(n)$ real numbers. Further, we
shall see numerically that the number of iterations required to
achieve a solution of (\ref{chapno1.2}) or (\ref{chapno1.3}) 
to within a given tolerance is
independent of $n$. The {\sc Matlab} software used can be obtained
from my homepage.

Our method applies to many other radial basis functions, such as the
inverse multiquadric ($\phi(x) = (x^2 + c^2)^{-1/2}$) and the thin
plate spline ($\phi(x) = x^2 \log |x|$). However, we concentrate on
the Gaussian and the multiquadric because they exhibit most of the
important features of our approach in a concrete setting. Similarly we
treat the one-dimensional problem merely to avoid complication; the
multidimensional case is a rather slight generalization of this work.
Let us remark that the analogue
of (\ref{chapno1.1}) is the operator
\be
 A_n^{(d)} = \left( \phi(j-k) \right)_{j,k\in [-n,n]^d},
\label{chapno1.4}\ee
and we shall still call $A_n^{(d)}$ a Toeplitz matrix. Moreover
the matrix-vector multiplication
\be A_n^{(d)} x = \left( \sum_{k \in [-n,n]^d} \phi(\|j-k\|) x_k
\right)_{j\in [-n,n]^d}, \label{chapno1.5}\ee
where $\|\cdot\|$ is the Euclidean norm and $x = (x_j)_{j\in [-n,n]^d}$, can
still be calculated in $\OO(N \log N)$ operations, where $N = (2n+1)^d$,
requiring $\OO(N)$ real numbers to be stored. This trick is a simple extension
of the Toeplitz matrix-vector multiplication method when $d = 1$.

\section{The Gaussian}
It is well-known that the Gaussian generates a positive definite
interpolation matrix, and its functional decay is so rapid that
preconditioning the conjugate gradient algorithm is not
necessary. However, it provides a useful model problem that we shall
describe here before developing the ideas further in the following
section.

Our treatment of the preconditioned conjugate gradient (PCG) method follows
Section 10.3 of Golub and Van Loan (1989), and we begin with a general
description. We let $n$ be a positive integer and $A \in \Rnn$ 
be an arbitrary symmetric
positive definite matrix. For any nonsingular symmetric matrix $P \in \Rnn$
and $b \in \RR^n$ we can use the following iteration to solve the linear system
$PAPx =Pb$.

\begin{algorithm}
Choose any $x_0$ in $\Rn$. Set $r_0 = Pb - PAP x_0$
and $d_0 = r_0$.

{\obeylines 
For $k=0, 1, 2, \ldots$ do begin
\quad $a_k = r_k^Tr_k / d_k^T PAP d_k$
\qquad $x_{k+1} = x_k + a_k d_k$
\qquad $r_{k+1} = r_k - a_k PAP d_k$
\quad $b_k = r_{k+1}^T r_{k+1} / r_k^T r_k$
\qquad $d_{k+1} = r_{k+1} + b_k d_k$
\quad Stop if $\|r_{k+1}\|$ or $\|d_{k+1}\|$ is sufficiently small.
end.\par}

\label{alg2.1}
\end{algorithm}

\noindent In order to simplify Algorithm \ref{alg2.1} define
\be
C = P^2, \qquad
\xi_k = P x_k, \qquad
r_k = P \rho_k \qquad
\hbox{ and } \qquad \delta_k = P d_k.
\label{chapno2.1}\ee
Substituting in Algorithm \ref{alg2.1} we obtain the following method.

\begin{algorithm} 
Choose any $\xi_0$ in $\Rn$. Set
$\rho_0 = b - A \xi_0$, $\delta_0 = C \rho_0$.

{\obeylines 
For $k=0, 1, 2, \ldots$ do begin
\quad $a_k = \rho_k^T C \rho_k / \delta_k^T A \delta_k$
\qquad $\xi_{k+1} = \xi_k + a_k \delta_k$
\qquad $\rho_{k+1} = \rho_k - a_k A \delta_k$
\quad $b_k = \rho_{k+1}^T C \rho_{k+1} / \rho_k^T C \rho_k$
\qquad $\delta_{k+1} = C \rho_{k+1} + b_k \delta_k$
\quad Stop if $\|\rho_{k+1}\|$ or $\|\delta_{k+1}\|$ is sufficiently small.
end.\par}

\label{alg2.2}
\end{algorithm}

It is Algorithm \ref{alg2.2} that we shall consider as our PCG method in this
section, and we shall call $C$ the preconditioner. We see that the
only restriction on $C$ is that it must be a symmetric positive
definite matrix, but we observe that the spectrum of $CA$ should
consist of a small number of clusters, preferably one cluster
concentrated at one. At this point, we also mention that the condition
number of $CA$ is not a reliable guide to the efficacy of our
preconditioner. For example, consider the two cases when (i) $CA$ has
only two different eigenvalues, say $1$ and $100,000$, and (ii) when
$CA$ has eigenvalues uniformly distributed in the interval $[1, 100]$.
The former has the larger condition number but, in exact
arithmetic, the answer will be achieved in two steps, whereas the
number of steps can be as high as $n$ in the latter case. Thus the
term ``preconditioner'' is sometimes inappropriate, although its usage
has become standard.

In this paper we concentrate on preconditioners for the Toeplitz matrices generated
by radial basis function interpolation on a (finite) regular grid.
Accordingly, we let $A$ be the matrix $A_n$ of (\ref{chapno1.1}) and 
let $\phi(x) = \exp(-
x^2)$. Thus $A_n$ is positive definite and can be embedded in the 
bi-infinite symmetric
Toeplitz matrix
\be A_\infty = \left(\phi(j-k) \right)_{j,k\in\ZZ}. \label{chapno2.2}\ee
The classical theory of Toeplitz operators (see, for instance,
Grenander and Szeg\H{o} (1984)) and the work of Baxter (1994) provide the
relations   
\be \Sp A_n \subset \Sp A_\infty = [\sigma(\pi), \sigma(0)] \subset (0,\infty),
\label{chapno2.3}\ee
where $\sigma$ is the symbol function
\be \sigma(\xi) = \sum_{k\in\ZZ} \phihat(\xi+2\pi k), \qquad \xi \in
\RR, \label{chapno2.4}\ee
and $\Sp A_\infty$ denotes the spectrum of the operator $A_\infty$.
Further, Theorem 9 of Buhmann and Micchelli (1991) allows us to
conclude that, for any fixed integers $j$ and $k$, we have
\be \lim_{n \to \infty} (A_n^{-1})_{j,k} = (A_\infty^{-1})_{j,k}.
\label{chapno2.5}\ee
It was equations (\ref{chapno2.3}) and (\ref{chapno2.5}) which led us to investigate the possibility of
using some of the elements of $A_n^{-1}$ for a relatively small value of $n$
to construct preconditioners for $A_N$, where $N$ may be much larger
than $n$. Specifically, let us choose integers $0 < m \le n$ and define the
sequence 
\be
c_j = (A_n^{-1})_{j0}, \qquad j = -m, \ldots, m. 
\label{chapno2.6}\ee
We now let $C_N$ be the $(2N+1) \times (2N+1)$ banded symmetric Toeplitz
matrix
\be C_N = \pmatrix{ c_0 & \ldots & c_m & & & \cr
                   \vdots & \ddots & & \ddots & & \cr
                   c_m & & & & & \cr
                   & \ddots & & & & c_m \cr
                   & & & & & \vdots \cr
                   & & & c_m & \ldots & c_0 \cr}. \label{chapno2.7}
\ee
We claim that, for sufficiently large $m$ and $n$, $C_N$ provides an
excellent preconditioner when $A = A_N$ in Algorithm \ref{alg2.2}. Before
discussing any theoretical motivation for this choice of
preconditioner, we present an example. We let $n=64$, $m=9$ and $N =
32,768$. Constructing $A_n$ and calculating the elements $\{
(A_n^{-1})_{j0}: j=0, 1, \ldots, m \}$ we find that
\be \pmatrix{ c_0 \cr c_1 \cr \vdots \cr c_9 } 
  = \pmatrix{    
  \ \,\,1.4301 \times 10^0 \cr
  -5.9563 \times 10^{-1}\cr
  \ \,\, 2.2265 \times 10^{-1}\cr
  -8.2083 \times 10^{-2}\cr
  \ \,\, 3.0205 \times 10^{-2}\cr
  -1.1112 \times 10^{-2}\cr
  \ \,\, 4.0880 \times 10^{-3}\cr
  -1.5039 \times 10^{-3}\cr
  \ \,\, 5.5325 \times 10^{-4}\cr
  -2.0353 \times 10^{-4}\cr}.
\label{chapno2.8}\ee

\begin{figure}[h]
\begin{center}
\psfig{file=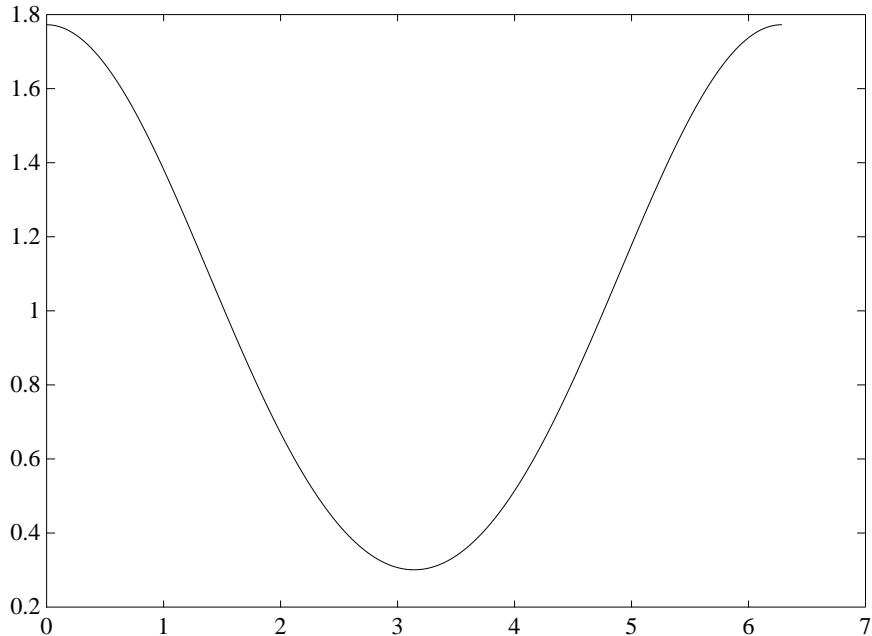,width=5in}
\caption{The symbol function for $C_\infty$.}
\label{fig1}
\end{center}
\end{figure}

Now $C_N$ can be embedded in the bi-infinite Toeplitz matrix
$C_\infty$ defined by
\be (C_\infty)_{jk} = \cases{ c_{j-k}, \qquad |j-k| \le m, \cr 0,
\qquad |j-k| > m,} \label{chapno2.9}\ee
and the symbol for this operator is the trigonometric polynomial
\be \sigma_{C_\infty}(\xi) = \sum_{j=-m}^m c_j e^{ij\xi}, \qquad \xi
\in \RR. \label{chapno2.10}\ee

In Figure \ref{fig1} we display a graph of $\sigma_{C_\infty}$ for $0 \le \xi
\le 2\pi$, and it is clearly a positive function. Thus the relations
\be \Sp C_N \subset \Sp C_\infty = \{\sigma_{C_\infty}(\xi) : \xi \in
[0,2\pi] \} \subset (0, \infty) \label{chapno2.11}\ee
imply that $C_N$ is positive definite. Hence it is suitable to use
$C_N$ as the preconditioner in Algorithm \ref{alg2.2}. Our aim in this example
is to compare this choice of preconditioner with the use of the
identity matrix as the preconditioner. To this end, we let the
elements of the vector $b$ of Algorithm \ref{alg2.2} be random
real numbers uniformly distributed in the interval $[-1,1]$. Applying
Algorithm \ref{alg2.2} using the identity matrix as the preconditioner provides
the results of Table \ref{table1}. 
Table \ref{table2} contains the analogous results
using (\ref{chapno2.7}) and (\ref{chapno2.8}). 
In both cases the iterations were stopped when
the residual vector satisfied the bound $\|r_{k+1}\| /\|b\| < 10^{-13}$.
The behaviour shown in the tables is typical; we find that the number of steps
required is independent of $N$ and $b$.

\begin{table}
\caption{No preconditioning}
\begin{tabular}{r|r}\hline\hline
{\rm Iteration}&{\rm Error}\\
\hline
$1$& $2.797904 \times 10^{1}$\\
$10$& $1.214777 \times 10^{-2}$\\
$20$& $1.886333 \times 10^{-6}$\\
$30$& $2.945903 \times 10^{-10}$\\
$33$& $2.144110 \times 10^{-11}$\\
$34$& $8.935534 \times 10^{-12}$\\
\hline
\end{tabular}
\label{table1}
\end{table}

\begin{table}
\caption{Using (\ref{chapno2.7}) and
(\ref{chapno2.8}) as the preconditioner}
\begin{tabular}{r|r}\hline\hline
{\rm Iteration}&{\rm Error}\\
\hline
$1$& $2.315776 \times 10^{-1}$\cr
$2$& $1.915017 \times 10^{-3}$\cr
$3$& $1.514617 \times 10^{-7}$\cr
$4$& $1.365228 \times 10^{-11}$\cr
$5$& $1.716123 \times 10^{-15}$\cr
\hline
\end{tabular}
\label{table2}
\end{table}

Why should (\ref{chapno2.7}) and (\ref{chapno2.8}) provide a good preconditioner? Let us
consider the bi-infinite Toeplitz matrix $C_\infty A_\infty$. The
spectrum of this operator is given by
\be \Sp C_\infty A_\infty = \{ \sigma_{C_\infty}(\xi) \sigma(\xi) : \xi
\in [0,2\pi] \}, \label{chapno2.12}\ee
where $\sigma$ is given by (\ref{chapno2.4}) and $\sigma_{C_\infty}$ by (\ref{chapno2.10}).
Therefore in order to concentrate $\Sp C_\infty A_\infty$ at unity we
must have
\be \sigma_{C_\infty}(\xi) \sigma(\xi) \approx 1, \qquad \xi \in
[0,2\pi]. \label{chapno2.13}\ee
In other words, we want $\sigma_{C_\infty}$ to be a trigonometric
polynomial approximating the continuous function $1/\sigma$. Now if
the Fourier series of $1/\sigma$ is given by
\be \sigma^{-1}(\xi) = \sum_{j\in\ZZ} \gamma_j e^{ij\xi}, \qquad \xi
\in \RR, \label{chapno2.14}\ee
then its Fourier coefficients $(\gamma_j)_{j\in\ZZ}$ are the
coefficients of the cardinal function $\chi$ for the integer grid,
that is
\be \chi(x) = \sum_{j\in\ZZ} \gamma_j \phi(x-j), \qquad x \in \RR,
\label{chapno2.15}\ee
and 
\be   \chi(k) = \delta_{0k}, \qquad k \in \ZZ. \label{chapno2.16}\ee
(See, for instance, Buhmann (1990).) Recalling (\ref{chapno2.5}), we deduce that
one way to calculate approximate values of the coefficients
$(\gamma_j)_{j\in\ZZ}$ is to solve the linear system
\be A_n c^{(n)} = e^0, \label{chapno2.17}\ee
where $e^0 = (\delta_{j0})_{j=-n}^n \in \RR^{2n+1}$. 
We now set
\be c_j = c_j^{(n)}, \qquad 0 \le j \le m, \label{chapno2.18}\ee
and we observe that the symbol function $\sigma$ for the Gaussian is a
theta function (see Baxter (1994), Section 2). Thus $\sigma$ is a positive
continuous function whose Fourier series is absolutely convergent.
Hence $1/\sigma$ is a positive continuous function and Wiener's lemma
(Rudin (1973)) implies the absolute convergence, and therefore the uniform
convergence, of its Fourier series. We deduce that the symbol function
$\sigma_{C_\infty}$ can be chosen to approximate $1/\sigma$ to within
any required accuracy. More formally we have the 

\begin{lem}
Given any $\epsilon > 0$, there are positive
integers $m$ and $n_0$ such that
\be \Bigl| \sigma(\xi)\sum_{j=-m}^m c_j^{(n)} e^{ij\xi} - 1 \Bigr| \le
\epsilon, \qquad \xi \in [0,2\pi],\ee
for every $n \ge n_0$, where $c^{(n)} = (c_j^{(n)})_{j=-n}^n$ is given
by (\ref{chapno2.17}).
\end{lem}

\begin{proof} 
The uniform convergence of the Fourier series for $\sigma^{-1}$ implies
that we can choose $m$ such that
\be 
\Bigl|\sigma(\xi)\sum_{j=-m}^m  \gamma_j e^{ij\xi} -1 \Bigr| \le
\epsilon, \qquad \xi \in [0,2\pi].
\ee
By (\ref{chapno2.5}), we can also choose $n_0$ such that 
$\max \{ |\gamma_j - c_j^{(n)}| : j = -m, \ldots, m\} \le \epsilon$,
when $n \ge n_0$.
Then we have
\begin{eqnarray}
\lefteqn{\Bigl| \sigma(\xi) \sum_{j=-m}^m c_j^{(n)} e^{ij\xi} - 1 \Bigr|}
\nonumber\\
   &\le& \Bigl| \sigma(\xi)\sum_{j=-m}^m \gamma_j e^{ij\xi} -1 \Bigr| +
 \Bigl| \sigma(\xi)
    \sum_{j=-m}^m (\gamma_j - c_j^{(n)}) e^{ij\xi}\Bigr| \nonumber\\
   &\le& \epsilon [ 1 + (2m+1)\|\sigma\|_\infty]. \nonumber\\
\end{eqnarray}
Since $\epsilon$ is arbitrary the proof is complete.
\end{proof}

\section{The Multiquadric}
The multiquadric interpolation matrix
\be
 A = \Bigl( \phi(\|x_j - x_k\|) \Bigr)_{j,k=1}^n, 
\ee
where $\phi(r) = (r^2 + c^2)^{1/2}$ and $(x_j)_{j=1}^n$ are points in
$\Rd$, is not positive definite. In Micchelli (1986),
it was shown to be
{\it almost negative definite}, that is for any real numbers
$(y_j)_{j=1}^n$ satisfying $\sum y_j = 0$ we have
\be
 \sum_{j,k=1}^n y_j y_k \phi(\|x_j - x_k\|) \le 0. \label{chapno3.1}\ee
Furthermore, inequality (\ref{chapno3.1}) is strict when $n \ge 2$, the
points $(x_j)_{j=1}^n$ are all different, 
and $\sum |y_j| > 0$. 
In other words, $A$ is negative definite on the subspace
$\langle e \rangle^\perp$, where $e = [1, 1, \ldots, 1]^T \in \RR^n$. 

Of course we cannot apply Algorithms \ref{alg2.1} and \ref{alg2.2} 
in this case. However,
we can use the almost negative definiteness of $A$ to solve a closely
related linearly constrained quadratic programming problem:
\begin{eqnarray}
 &\ &\hbox{ minimize } \quad \half \xi^T A \xi - b^T \xi \nonumber\\
 &\ &\hbox{ subject to } \quad e^T \xi = 0, \nonumber\\
\label{chapno3.2}
\end{eqnarray}
where $b$ can be any element of $\Rn$. 
Standard theory of Lagrange
multipliers guarantees the existence of a unique pair of vectors $\xi^* \in \RR^n$
and $\eta^* \in \RR^m$ satisfying the equations
\begin{eqnarray}
 A \xi^* + e \eta^* &=& b, \nonumber\\
          e^T \xi^* &=& 0,
\label{chapno3.3}
\end{eqnarray}
where $\eta^*$ is the Lagrange multiplier vector for the constrained
optimization problem (\ref{chapno3.2}). We do not go into further detail on this point
because the nonsingularity of the matrix
\be
 \pmatrix{ A & e \cr
e^T & 0} \label{chapno3.4}\ee
is well-known (see, for
instance, Powell (1990)). Instead we observe that one way to solve
(\ref{chapno3.3}) is to apply the following modification of Algorithm
\ref{alg2.1} 
to (\ref{chapno3.2}).

\begin{algorithm}
Let $P$ be any symmetric $n \times n$
matrix such that $\ker P = \langle e \rangle$.

{\obeylines 
Set $x_0 = 0$, $r_0 = Pb - PAP x_0$, $d_0 = r_0$.
For $k=0, 1, 2, \ldots$ do begin
\quad $a_k = r_k^Tr_k / d_k^T PAP d_k$
\qquad $x_{k+1} = x_k + a_k d_k$
\qquad $r_{k+1} = r_k - a_k PAP d_k$
\quad $b_k = r_{k+1}^T r_{k+1} / r_k^T r_k$
\qquad $d_{k+1} = r_{k+1} + b_k d_k$
\quad Stop if $\|r_{k+1}\|$ or $\|d_{k+1}\|$ is sufficiently small.
\quad end.\par}

\label{alg3.1}
\end{algorithm}

We observe that Algorithm \ref{alg3.1} solves the linearly constrained
optimization problem
\begin{eqnarray}
 &\ &\hbox{ minimize } \quad \half x^T PAP x - b^T Px \nonumber\\
 &\ &\hbox{ subject to } \quad e^T x = 0. \nonumber\\
\label{chapno3.5}
\end{eqnarray}
Moreover, the following elementary lemma implies that the solutions $\xi^*$of
(\ref{chapno3.3}) and $x^*$ of (\ref{chapno3.5}) are related by the equations $\xi^* = P x^*$.

\begin{lem}
Let $S$ be any symmetric $n \times n$ matrix
and let $K = \ker S$. Then $S : K^\perp \to K^\perp$ is a bijection.
In other words, given any $b \in K^\perp$ there is precisely one $a \in
K^\perp$ such that
\be 
S a = b. 
\label{chapno3.6}\ee
\label{lem3.2}
\end{lem}

\begin{proof} 
For any $n \times n$ matrix $M$ we have the equation
\[
 \RR^n = \ker M \oplus \Image M^T.
\]
Consequently the symmetric matrix $S$ satisfies
\[ 
\RR^n = \ker S \oplus \Image S, 
\]
whence $\Image S = K^\perp$. Hence for every $b \in K^\perp$ there exists
$\alpha \in \Rn$ such that $S \alpha = b$. Now we can write $\alpha =
a + \beta$, where $a \in K^\perp$ and $\beta \in K$ are uniquely
determined by $\alpha$. Thus $S a = S \alpha = b$, and (\ref{chapno3.6}) has a
solution. If $a^\prime \in K^\perp$ also satifies (\ref{chapno3.6}), then their
difference $a - a^\prime$ lies in the intersection $K \cap K^\perp =
\{0\}$, which settles the uniqueness of $a$.
\end{proof}

Setting $P = S$ and $K = \langle e \rangle$ in Lemma \ref{lem3.2} 
we deduce that there is
exactly one $x^* \in \langle e \rangle^\perp$ such that
\[
 PAP x^* = Pb, 
\]
and $PAP$ is negative definite when restricted to the subspace
$\langle e \rangle^\perp$.
Following the development of Section 2, we define
\be
C = P^2, \qquad
\xi_k = Px_k, \qquad
\hbox{ and } \qquad \delta_k = P d_k,\label{chapno3.7}\ee
as in equation (\ref{chapno2.1}).
However, we cannot define $\rho_k$ by (\ref{chapno2.1}) because $P$ is singular.
One solution, advocated by Dyn, Levin and Rippa (1986), is to use the
recurrence for $(\rho_k)$ embodied in Algorithm \ref{alg2.1} without further
ado.

\begin{algorithm}
Choose any $\xi_0$ in $\langle e \rangle^\perp$. Set
$\rho_0 = b - A \xi_0$ and $\delta_0 = C \rho_0$.

{\obeylines 
For $k=0, 1, 2, \ldots$ do begin
\quad $a_k = \rho_k^T C \rho_k / \delta_k^T A \delta_k$
\qquad $\xi_{k+1} = \xi_k + a_k \delta_k$
\qquad $\rho_{k+1} = \rho_k - a_k A \delta_k$
\quad $b_k = \rho_{k+1}^T C \rho_{k+1} / \rho_k^T C \rho_k$
\qquad $\delta_{k+1} = C \rho_{k+1} + b_k \delta_k$
\quad Stop if $\|\rho_{k+1}\|$ or $\|\delta_{k+1}\|$ is sufficiently small.
end.\par}

\label{alg3.3a}
\end{algorithm}

However this algorithm is unstable in finite precision arithmetic, as
we shall see in our main example below. One modification that
successfully avoids instability is to force the condition 
\be\rho_k \in \langle e \rangle^\perp \label{chapno3.8}\ee
to hold for all $k$.
Now Lemma \ref{lem3.2} implies the existence of exactly one vector
$\rho_k \in \langle e \rangle^\perp$ for which $P \rho_k = r_k$. Therefore,
defining $Q$ to be the orthogonal projection onto $\langle e \rangle^\perp$, that
is $Q : x \mapsto x - e (e^T x)/(e^T e)$, we obtain

\begin{algorithm}
Choose any $\xi_0$ in $\langle e \rangle^\perp$. Set
$\rho_0 = Q(b - A \xi_0)$, $\delta_0 = C \rho_0$.

{\obeylines 
For $k=0, 1, 2, \ldots$ do begin
\quad $a_k = \rho_k^T C \rho_k / \delta_k^T A \delta_k$
\qquad $\xi_{k+1} = \xi_k + a_k \delta_k$
\qquad $\rho_{k+1} = Q(\rho_k - a_k A \delta_k)$
\quad $b_k = \rho_{k+1}^T C \rho_{k+1} / \rho_k^T C \rho_k$
\qquad $\delta_{k+1} = C \rho_{k+1} + b_k \delta_k$
\quad Stop if $\|\rho_{k+1}\|$ or $\|\delta_{k+1}\|$ is sufficiently small.
end.\par}
\label{alg3.3b}
\end{algorithm}

We see that the only restriction on $C$ is that it must be a
non-negative definite symmetric matrix such that $\ker C = \langle e \rangle$. It is
easy to construct such a matrix given a positive definite symmetric
matrix D by a rank one modification:
\be C = D - {(De)(De)^T \over e^T De}. \label{chapno3.9}\ee
The Cauchy-Schwarz inequality implies that $x^T C x \ge 0$ with
equality if and only if $x \in \langle e \rangle$. Of course we do not need to form
$C$ explicitly, since $C : x \mapsto Dx - (e^T Dx / e^T De) De$.
Before constructing $D$ we consider the spectral properties of
$A_\infty = (\phi(j-k))_{j,k\in\ZZ}$ in more detail.

A minor
modification to Proposition 5.2.2 of Baxter (1992)
yields the following useful result.
Let us say that a complex sequence $(y_j)_\ZZ$ is {\em zero-summing} if
it is finitely supported and satisfies $\sum y_j = 0$. The {\em symbol
function}
\be
\sigma(\xi) = \sum_{k\in\ZZ} \phihat(\xi+2\pi k), \qquad \xi \in \RR,
\ee
now requires the distributional Fourier transform of the
multiquadric. In the univariate case, this is given by
\be
\phihat(\xi) = -(2 c/|\xi|) K_1(c|\xi|), \qquad\xi \in
\RR\setminus\{0\},
\ee
where $K_1$ is a modified Bessel function. The symbol function is
studied extensively in Baxter (1994).

\begin{figure}[h]
\begin{center}
\psfig{file=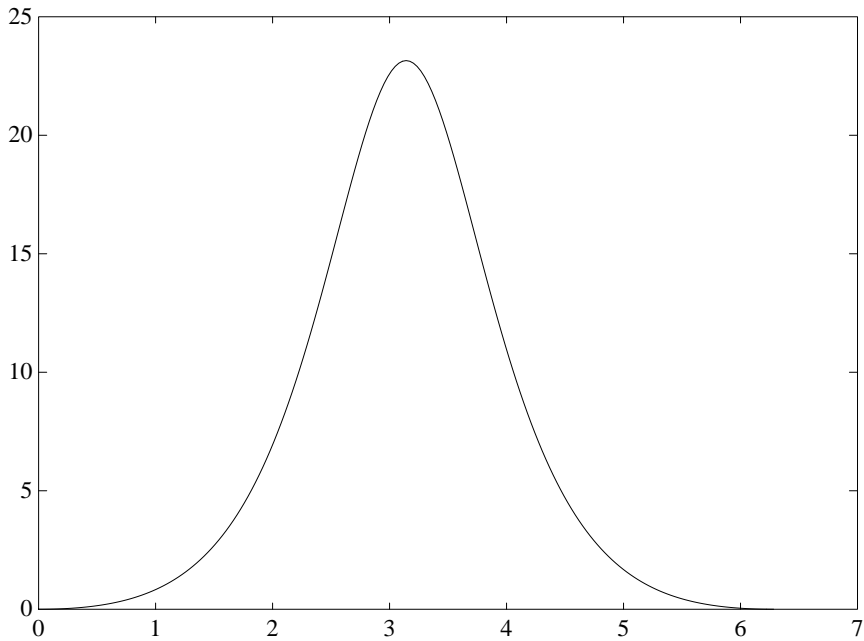,width=5in}
\caption{The reciprocal symbol function $1/\sigma$ for
the multiquadric.}
\label{fig2}
\end{center}
\end{figure}

\begin{prop}
For every $\eta \in (0,2\pi)$ we
can find a set $\{ (y_j^{(n)})_{j\in\ZZ} : n = 1, 2, \ldots \}$ of
zero-summing sequences such that
\be 
\lim_{n\to\infty} \sum_{j,k\in\ZZ} y_j^{(n)} \overline{y_k^{(n)}} \phi(j-k)
\Bigl/ \sum_{j\in\ZZ} |y_j^{(n)}|^2 = \sigma(\eta). 
\label{chapno3.10}
\ee

\label{prop3.4}
\end{prop}

\begin{proof} 
We adopt the proof technique of Proposition 5.2.2 of Baxter (1992). 
For each
positive integer $n$ we define the trigonometric polynomial
\[
L_n(\xi) = n^{-1/2} \sum_{k=0}^{n-1} e^{ik\xi}, \qquad \xi \in
\RR,
\]
and we recall from Section 2 of Baxter (1994) that 
\be
 K_n(\xi) = {\sin^2 n\xi/2 \over n \sin^2 \xi/2} = \left| L_n(\xi)
\right|^2, \label{chapno3.11}\ee
where $K_n$ is the $n$th degree Fej\'er kernel. We now choose
$(y^{(n)}_j)_{j\in\ZZ}$ to be the Fourier coefficients of the
trigonometric polynomial $\xi \mapsto L_n(\xi - \eta) \sin \xi/2$,
which implies the relation
\[
 \Bigl| \sum_{j\in\ZZ} y^{(n)}_j e^{ij\xi} \Bigr|^2 = \sin^2 \xi/2\ 
K_n(\xi-\eta), \]
and we see that $(y_j^{(n)})_{j\in\ZZ}$ is a zero-summing sequence. By
the Parseval relation we have
\be 
\sum_{j\in\ZZ} |y_j^{(n)}|^2 = (2\pi)^{-1} \int_0^{2\pi} \sin^2
\xi/2 \ K_n(\xi - \eta)\,d\xi 
\label{chapno3.12}\ee
and the approximate identity property of the Fej\'er kernel (Zygmund
(1988), p. 86) implies that
\[
 \sin^2 \eta/2 
  = \lim_{n\to\infty} (2\pi)^{-1} \int_0^{2\pi} \sin^2
\xi/2 \ K_n(\xi - \eta)\,d\xi 
  = \lim_{n\to\infty} \sum_{j\in\ZZ} |y_j^{(n)}|^2.
\]
Further, because $\sigma$ is continuous on $(0,2\pi)$ (Baxter (1994),
Section 4.4), we have 
\begin{eqnarray*}
 \sin^2 \eta/2 \ \sigma(\eta)
  &=& \lim_{n\to\infty} (2\pi)^{-1} \int_0^{2\pi} \sin^2
\xi/2 \ K_n(\xi - \eta) \sigma(\xi)\,d\xi \\
  &=& \lim_{n\to\infty} 
       \sum_{j,k\in\ZZ} y^{(n)}_j \overline{y^{(n)}_k} \phi(j-k). \\
\end{eqnarray*}
\end{proof}

Thus we have shown that, just as in the classical theory of Toeplitz
operators (Grenander and Szeg\H{o} (1984)), everything depends on the
range of values of the symbol function $\sigma$. Because $\sigma$
inherits the double pole that $\phihat$ enjoys at zero, we have
$\sigma \colon (0,2\pi) \mapsto (\sigma(\pi), \infty)$. In Figure
\ref{fig2} we display the function $\sigma^{-1}$.

Now let $m$ be a positive integer and let $(d_j)_{j=-m}^m$ be an even
sequence of real numbers. We define a bi-infinite banded symmetric
Toeplitz matrix $D_\infty$ by the equations
\be (D_\infty)_{jk} = \cases{ d_{j-k},  \qquad |j-k| \le m, \cr 0,
\qquad \hbox{ otherwise }.} \label{chapno3.13}\ee
Thus $(D_\infty A_\infty)_{jk} = \psi(j-k)$ where $\psi(x) =
\sum_{l=-m}^m d_l \phi(x-l)$. Further
\be \sum_{j,k\in\ZZ} y_j \overline{y_k} \psi(j-k) 
  = (2\pi)^{-1} \int_0^{2\pi} \Bigl| \sum_{j\in\ZZ} y_j e^{ij\xi}
\Bigr|^2 \sigma_{D_\infty}(\xi) \sigma(\xi)\,d\xi,
\label{chapno3.14}\ee
where the symbol function $\sigma_{D_\infty}$ for the Toeplitz
operator $D_\infty$ is given by
\be
\sigma_{D_\infty}(\xi) = \sum_{j=-m}^m d_j e^{ij\xi}, \xi \in \RR.
\ee
Now the function $\sigma \sigma_{D_\infty}$ is
continuous for $\xi \in (0,2\pi)$, so the argument of Proposition
\ref{prop3.4} also shows that, for every $\eta \in (0,2\pi)$, we can find
a set $\{ (y^{(n)}_j)_{j\in\ZZ} : n = 1, 2, \ldots\ \}$ of
zero-summing sequences such that
\be 
\lim_{n\to\infty} 
\frac{\sum_{j,k\in\ZZ} y^{(n)}_j \overline{y^{(n)}_k} \psi(j-k)}
 {\sum_{j\in\ZZ} |y^{(n)}_j|^2} 
= \sigma_{D_\infty}(\eta)\sigma(\eta). 
\label{chapno3.15}\ee

A good preconditioner must ensure that $\{
\sigma_{D_\infty}(\xi) \sigma(\xi) : \xi \in (0,2\pi) \}$ is a bounded
set. Because of the form of $\sigma_{D_\infty}$ we have the equation
\be \sum_{j=-m}^m d_j = 0. \label{chapno3.16}\ee
Moreover, as in Section 2, we want the approximation
\be \sigma_{D_\infty}(\xi) \sigma(\xi) \approx 1, \qquad \xi \in
(0,2\pi), \label{chapno3.17}\ee
and we need $\sigma_{D_\infty}$ to be a non-negative trigonometric
polynomial which is positive almost everywhere, which ensures that
every one of its principal minors is positive definite.

Let us define
\be
c_j^{(n)} = -\left(A_n^{-1}\right)_{j0}, \qquad j = -m, \ldots, m,
\label{3.15a}
\ee
and 
\be
\sigma^{-1}(\xi) = \sum_{j\in\ZZ} \gamma_j e^{ij\xi}, \qquad \xi \in
\RR.
\label{3.15b}
\ee
Then Theorem 9 of Buhmann and Micchelli (1991) states that 
\be
\lim_{n \to \infty} c_j^{(n)} = \gamma_j,
\label{3.15c}
\ee
for any given fixed integer $j$. We shall use this fact to construct a
suitable $\sigma_{D_\infty}$. First we subtract a multiple of the
vector $[1, \ldots, 1]^T \in \RR^{2m+1}$ from $(c_j^{(n)})_{j=-m}^m$ to
form a new vector $(d_j)_{j=-m}^m$ satisfying $\sum d_j = 0$, and we
observe that, by (\ref{3.15c}), $\sigma_{D_{\infty}}(\xi)$ is one-signed
for all sufficiently large values of $n$.
For the numerical experiments here, we have chosen
$n=64$ and $m=9$. 

\begin{table}
\caption{Preconditioned CG -- $m=9$, $n=64$,
$N=2,048$}
\begin{tabular}{r|r}\hline\hline
{\rm Iteration}&{\rm Error}\cr
\hline
$1$& $3.975553 \times 10^{4}$\cr
$2$& $8.703344 \times 10^{-1}$\cr
$3$& $2.463390 \times 10^{-2}$\cr
$4$& $8.741920 \times 10^{-3}$\cr
$5$& $3.650521 \times 10^{-4}$\cr
$6$& $5.029770 \times 10^{-6}$\cr
$7$& $1.204610 \times 10^{-5}$\cr
$8$& $1.141872 \times 10^{-7}$\cr
$9$& $1.872273 \times 10^{-9}$\cr
$10$& $1.197310 \times 10^{-9}$\cr
$11$& $3.103685 \times 10^{-11}$\cr
\hline
\end{tabular}
\label{table3}
\end{table}

\begin{table}
\caption{Preconditioned CG -- $m=9$, $n=64$, $N=32,768$}
\begin{tabular}{r|r}\hline\hline
{\rm Iteration}&{\rm Error}\\
\hline
 $1$& $2.103778 \times 10^{5}$\cr
 $2$& $4.287497 \times 10^{0}$\cr
 $3$& $5.163441 \times 10^{-1}$\cr
 $4$& $1.010665 \times 10^{-1}$\cr
 $5$& $1.845113 \times 10^{-3}$\cr
 $6$& $3.404016 \times 10^{-3}$\cr
 $7$& $3.341912 \times 10^{-5}$\cr
 $8$& $6.523212 \times 10^{-7}$\cr
 $9$& $1.677274 \times 10^{-5}$\cr
 $10$& $1.035225 \times 10^{-8}$\cr
 $11$& $1.900395 \times 10^{-10}$\cr
\hline
\end{tabular}
\label{table4}
\end{table}

Thus, given 
\[ 
A_N = \Bigl( \phi(j-k) \Bigr)_{j,k=-N}^N
\]
for any $N \ge n$, we let $D_N$ be any $(2N+1) \times (2N+1)$ principal
submatrix of $D_\infty$ and define the preconditioner $C_N$ by the equation
\be C_N = D_N - {(D_N e) (D_N e)^T \over e^T D_N e},
\label{chapno3.22}\ee
where $e = [1, \ldots, 1]^T \in \RR^{2N+1}$. We reiterate that we
actually compute the matrix-vector product $C_N x$ by the operations
$x \mapsto D_N x - (e^T D_N x/e^T D_N e) e$ rather than by storing the
elements of $C_N$ in memory.

$C_N$ provides an excellent preconditioner. Tables \ref{table3} and 
\ref{table4}
illustrate its use when Algorithm \ref{alg3.3b} is applied to the linear
system
\begin{eqnarray}
 A_N x + e y &=& b, \nonumber\\
       e^T x &=& 0,\nonumber\\
\label{chapno3.23}
\end{eqnarray}
when $N=2,048$ and $N = 32,768$ respectively.
Here $y \in \RR$, $e = [1, \ldots, 1]^T \in \RR^{2N+1}$ and $b \in
\RR^{2N+1}$ consists of pseudo-random real numbers uniformly
distributed in the interval $[-1,1]$. Again, this behaviour is typical
and all our numerical experiments indicate that the number of steps is
independent of $N$. We remind the reader that the error shown is
$\|\rho_{k+1}\|$, but that the iterations are
stopped when either $\|\rho_{k+1}\|$ or $\|\delta_{k+1}\|$ is less
than $10^{-13} \|b\|$, where we are using the notation of Algorithm \ref{alg3.3b}.

\begin{figure}[h]
\begin{center}
\psfig{file=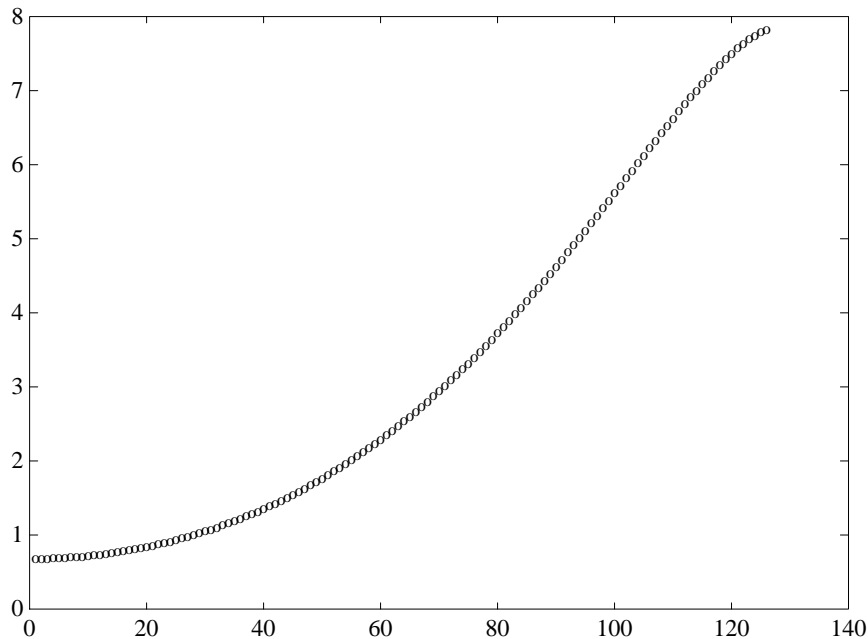,width=5in}
\caption{The spectrum of $C_n A_n$ for $m=1$
and $n=64$.}
\label{fig5}
\end{center}
\end{figure}

\begin{figure}[h]
\begin{center}
\psfig{file=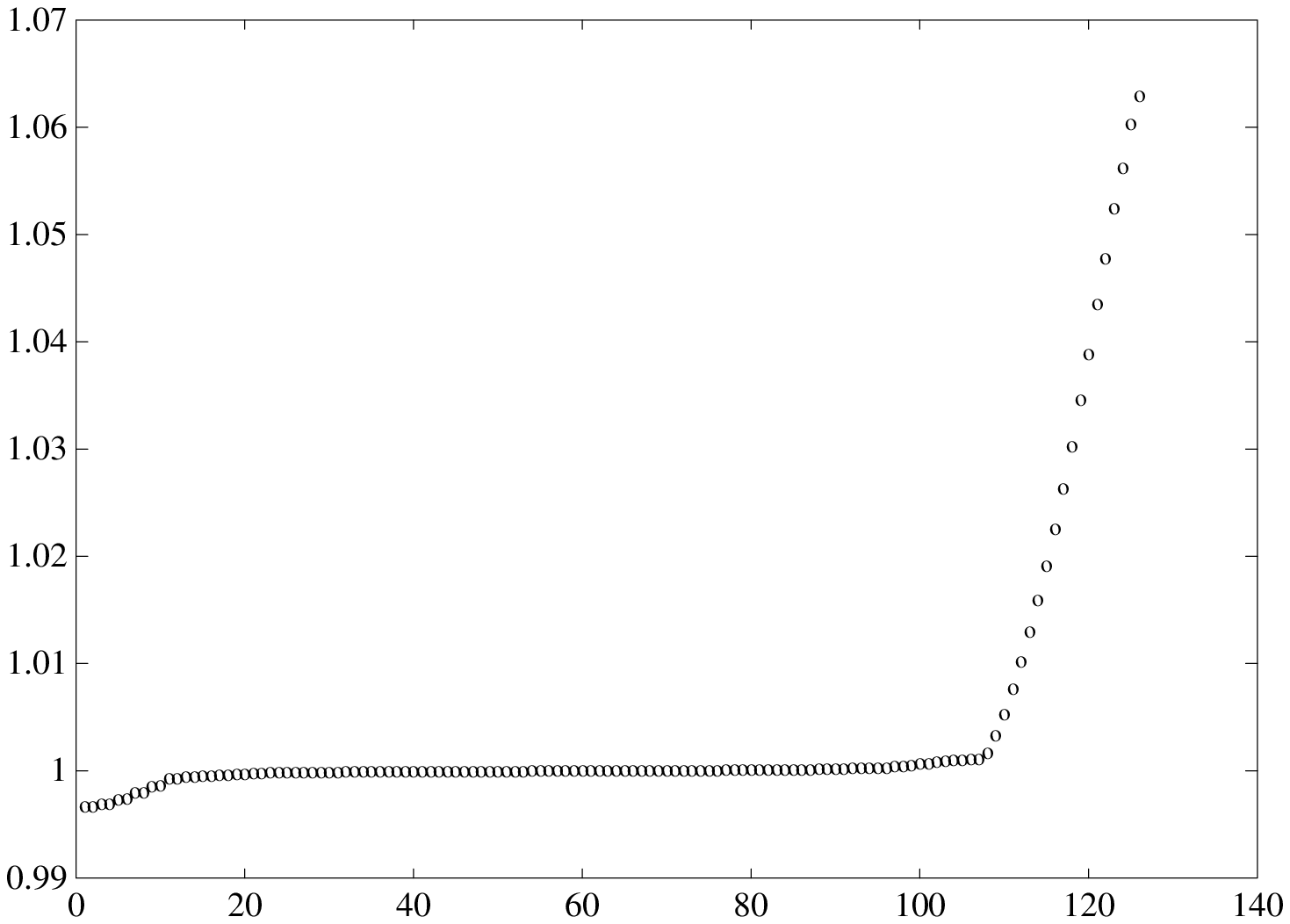,width=5in}
\caption{The spectrum of $C_n A_n$ for
$m=9$ and $n=64$.}
\label{fig6}
\end{center}
\end{figure}

\begin{table}
\caption{Preconditioned CG -- $m=1$, $n=64$, $N=8,192$}
\begin{tabular}{r|r}\hline\hline
{\rm Iteration}&{\rm Error}\cr
\hline
 $1$& $2.645008 \times 10^{4}$\cr
 $10$& $8.632419 \times 10^{0}$\cr
 $20$& $9.210298 \times 10^{-1}$\cr
 $30$& $7.695337 \times 10^{-1}$\cr
 $40$& $3.187051 \times 10^{-5}$\cr
 $50$& $5.061053 \times 10^{-7}$\cr
 $60$& $7.596739 \times 10^{-9}$\cr
 $70$& $1.200700 \times 10^{-10}$\cr
 $73$& $3.539988 \times 10^{-11}$\cr
 $74$& $1.992376 \times 10^{-11}$\cr
\hline
\end{tabular}
\label{table5}
\end{table}

\begin{table}
\caption{Algorithms  3.3a \& b -- $m=1$, $n=64$,
$N=64$, $b = [1, 4, \ldots, N^2]^T$.}
\begin{tabular}{r|rr}\hline\hline
{\rm Iteration}&{$\|\delta_k\|$ }-- 3.3a&{$\|\delta_k\|$ }-- 3.3b \cr
\hline
 $1$& $4.436896 \times 10^{4}$& $4.436896 \times 10^{4}$\cr
 $2$& $2.083079 \times 10^{2}$& $2.083079 \times 10^{2}$\cr
 $3$& $2.339595 \times 10^{0}$& $2.339595 \times 10^{0}$\cr
 $4$& $1.206045 \times 10^{-1}$& $1.206041 \times 10^{-1}$\cr
 $5$& $1.698965 \times 10^{-3}$& $1.597317 \times 10^{-3}$\cr
 $6$& $6.537466 \times 10^{-2}$& $6.512586 \times 10^{-2}$\cr
 $7$& $1.879294 \times 10^{-4}$& $9.254943 \times 10^{-6}$\cr
 $8$& $2.767714 \times 10^{-2}$& $1.984033 \times 10^{-7}$\cr
 $9$& $3.453789 \times 10^{-4}$\cr
 $10$& $1.914126 \times 10^{-3}$\cr
 $20$& $4.628447 \times 10^{-1}$\cr
 $30$& $3.696474 \times 10^{-0}$\cr
 $40$& $8.061922 \times 10^{+3}$\cr
 $50$& $2.155310 \times 10^{0}$\cr
 $100$& $3.374467 \times 10^{-1}$\cr
\hline
\end{tabular}
\label{table6}
\end{table}

It is interesting to compare Table \ref{table3} with Table \ref{table5}.
Here we have chosen $m=1$, and the preconditioner is essentially a multiple
of the second
divided difference preconditioner advocated by Dyn, Levin and Rippa
(1986). Indeed, we find that $d_0 = 7.8538$ and $d_1 = d_{-1} =
-3.9269$. 
We see that its behaviour is clearly inferior to the preconditioner
generated by choosing $m=9$. Furthermore, this is to be expected,
because we are choosing a smaller finite section to approximate the
reciprocal of the symbol function. However, because
$\sigma_{D_\infty}(\xi)$ is a multiple of $\sin^2 \xi/2$, this preconditioner
still possesses the property that $\{ \sigma_{D_\infty}(\xi)
\sigma(\xi) : \xi \in (0,2\pi) \}$ is a bounded set of real numbers.

It is also interesting to compare the spectra of $C_n A_n$ for $n=64$
and $m= 1$ and $m = 9$. Accordingly, Figures \ref{fig5} and \ref{fig6}
display all but the largest nonzero eigenvalues of $C_n A_n$ for $m=1$ and
$m=6$ respectively. The largest eigenvalues are $502.6097$ and
$288.1872$, respectively, and these were omitted from the plots
in order to reveal detail at smaller scales. We see that the clustering of the spectrum
when $m=9$ is excellent.  

The final topic in this section demonstrates the instability of
Algorithm \ref{alg3.3a} when compared with Algorithm \ref{alg3.3b}. We refer the
reader to Table \ref{table6}, where we have chosen $m=9$,
$n=N=64$, and setting $b = [1, 4, 9, \ldots, N^2]^T$.

\noindent
The iterations
for Algorithm \ref{alg3.3b}, displayed in Table \ref{table6}, were stopped at
iteration $108$. For Algorithm \ref{alg3.3a}, iterations were stopped when
either
$\|\rho_{k+1}\|$ or $\|\delta_{k+1}\|$ became smaller than $10^{-13}
\|b\|$. It is useful to display the norm of $\|\delta_k\|$ rather than
$\|\rho_k\|$ in this case. We see that the two algorithms almost agree
on the early interations, but that Algorithm \ref{alg3.3a} soon begins
cycling, and no convergence seems to occur. Thus when $\rho_k$ can
leave the required subspace due to finite precision arithmetic, it is
possible to attain non-descent directions.

\end{document}